\documentclass[preprint,12pt]{elsarticle}
\usepackage{xcolor}
\usepackage{amsbsy}
\usepackage{amscd}
\usepackage{amsfonts}
\usepackage{amsmath}
\usepackage{amssymb}
\usepackage{amstext}
\usepackage{amsthm}
\usepackage{epsfig}
\usepackage{fancybox}
\usepackage{graphicx}
\usepackage{latexsym}

\newtheorem{theorem}{Theorem}[section]
\newtheorem{lemma}[theorem]{Lemma}

\newtheorem{corollary}[theorem]{Corollary}
\theoremstyle{definition}
\newtheorem{definition}[theorem]{Definition}

\theoremstyle{remark}
\newtheorem{remark}[theorem]{Remark}
\numberwithin{equation}{section}

\begin{document}
\title{An It\^o-Wentzell formula for rough paths}
\author[1]{Rafael A. Castrequini\corref{cor1}}
\ead{rafael.castreq@gmail.com}

\author[2]{Pedro J. Catuogno}
\ead{pedrojc@ime.unicamp.br}

\author[3]{Alvaro E. Machado Hernandez}
\ead{alvaro 16841@hotmail.com}

\address[1]{Centro de Investigaci\'on y Modelamiento de Fen\'omenos Aleatorios, CIMFAV,
Universidad de Valpara\'iso, Valpara\'iso, Chile}

\address[2]{Departamento de Matem\'atica, IMECC,
Universidade Estadual de Campinas, 13083-859, Campinas - SP, Brasil}

\address[3]{Departamento de Matem\'atica, IMECC,
Universidade Estadual de Campinas, 13083-859, Campinas - SP, Brasil}
\cortext[cor1]{Corresponding author}

\begin{abstract}
This article shows an It\^o-Wentzell type formula adapted to rough paths with $\alpha$-H\"older regularity $\alpha \in (\frac{1}{3},\frac{1}{2}]$. We improve previous results of R. Castrequini and P. Catuogno \cite{Castrequini Catuogno} for the Young integral and C. Keller and J. Zhang \cite{Keller Zhang} for rough paths. 
\end{abstract}


\maketitle









\section{Introduction}
In 1950 K. It\^o \cite{Ito} published a formula for the derivation of a deterministic function whose argument is a solution of a stochastic differential equation, subsequently called It\^o formula.  The extension of this formula from deterministic functions to random fields of It\^o type was obtained by A. Wentzell \cite{Wentzell} in 1965.

Later several generalizations of the It\^o-Wentzell formula were obtained by several authors we mention between others the following: M. Bismut \cite{Bismut}, H. Kunita \cite{Kunita3}, B. Rozovskii \cite{Rozovskii} to the context of semimartingales and stochastic flows, N. Krylov \cite{Krylov} to Schwartz distributions,  D. Ocone and E. Pardoux \cite{Ocone Pardoux} to anticipating stochastic calculus,  R. Castrequini and P. Catuogno \cite{Castrequini Catuogno} to Young integral and R. Buckdhan, J. Ma  and  J. Zhang  \cite{Buckdhan Ma} to the pathwise functional It\^ o calculus.

In 1998 T. Lyons \cite{Lyons} introduced the rough path theory. It is a generalization of the ordinary differential equation theory to the case when the equation is driven by an  irregular signal, see T. Lyons  \cite{Lyons}, M. Gubinelli \cite{Gubinelli}, P. Friz and N. Victoir \cite{Friz} and P. Friz and M. Hairer \cite{Friz Hairer}. The rough path theory is powerful enough to include in its scope the stochastic differential equations.  In 2004, M. Gubinelli \cite{Gubinelli} gave a new approach to the theory of rough paths with the introduction of  the controlled paths. The controlled paths are the natural integrands of the rough integral, see \cite{Gubinelli} and \cite{Friz Hairer}. It is important to highlight  that the space of controlled paths is a Banach space closed by integration with respect to rough paths, composition with smooth functions and solutions of rough differential equation. From this viewpoint the rough path theory is formally similar to the classical semimartingale theory.

It is natural to ask for an It\^o-Wentzel type formula in the context of rough paths. In 2016, C. Keller and J. Zhang  \cite{Keller Zhang} gave a version of the It\^o-Wentzell formula for rough paths, see formula (3.14) in \cite{Keller Zhang}. This formula applies to controlled paths given by a rough integral under restrictive hypotesis on the vector field $g$.  In this article we give a complete proof of a It\^o-Wentzell type formula in the context of rough paths in a general setting. More precisely,  Let $\alpha \in (\frac{1}{3}, \frac{1}{2}]$,  $\mathbf{X}=(X, \mathbb{X})$ be an $\alpha$-rough path and $(h, \partial_Xh)$ be a controlled path. Let $g$ to be defined by  

\[
 g(t,x) = g(0,x) + \int_0^t h(s,x)d\mathbf{X}_s.
\]
Assume that $g$ is twice differentiable in the variable $x$ and the functions  $(t, x) \mapsto  D g_t(x)$  and $(t, x) \mapsto  D^2 g_t(x)$ are continuous.
Then under certain condition to be explained in Section 4, we have that for any controlled path $(Z, \partial_XZ)$ holds the following It\^o-Wentzell formula,
\begin{eqnarray}
 g(t,Z_t) & = & g(0,Z_0) + \int_0^t h(r,Z_r)d\mathbf{X}_r + \int_0^t Dg(r,Z_r) d_{\mathbf{X}}Z_r+\nonumber \\
 & &  + \int_0^ tDh(r,Z_r)\circ \partial_X Z_rd[\mathbf{X}]_{r} + \nonumber \\ & & +\frac{1}{2}\int_0^tD^ 2g(r, Z_r)  \partial_XZ_r \otimes \partial_X Z_r d[ \mathbf{X} ]_{r}, \nonumber
\end{eqnarray}
where $[\mathbf{X}]$ is the bracket of $\mathbf{X}$.

We would like to express the strong influence on our work of H. Kunita, see \cite{Kunita1}, \cite{Kunita3} and  \cite{Kunita2}.

The plan of exposition is as follows: In Section \ref{2}  we review some of the standard facts on rough paths and controlled paths in the Gubinelli's approach, see \cite{Friz Hairer} and \cite{Gubinelli}.

In Section \ref{3}, we will be concerned with a complete proof of a chain rule for controlled paths. 

In Section \ref{4}, we prove the It\^o-Wentzell formula for rough paths and show the relation with the result of C. Keller and J. Zhang \cite{Keller Zhang}. We give 
an adaptation of  H. Kunita's result about the decomposition 
 of solutions of stochastic differential 
equations to rough differential equations, see  \cite{Kunita1}. Finally, we study the Cauchy problem
for first-order semi-linear rough differential equations via the method of characteristics, in the spirit of H. Kunita \cite{Kunita2} and N. Krylov and B. Rozovskii \cite{Krylov Rozo}. 

In Appendix, we give the proof of some technical results necessary for the proof of Theorem 4.1.
             
\section{Rough paths and Controlled paths } \label{2}

In this section we present the elements of rough path theory. We shall use freely concepts and notations of P. Friz and M. Hairer \cite{Friz Hairer}. 

Let $V$, $U$ and $W$ be Banach spaces, $\alpha\in (\frac{1}{3}, \frac{1}{2}]$ and $T>0$. We will denote by $\triangle_T$ the set $\{(s,t): 0\leq s \leq t \leq T\}$.

\begin{definition}
An $\alpha$-H\"older rough path on $V$ is a pair $\mathbf{X}=(X, \mathbb{X})$ where $X : [0,T] \rightarrow V$ and $\mathbb{X}: \
\triangle_T\rightarrow V \otimes V$ such that 
\[
\|X\|_{\alpha} :=\sup_{s , t}\frac{|X_t -X_s|}{|t-s|^{\alpha}}< \infty ~;~\| \mathbb{X} \|_{2 \alpha}:= \sup_{s,t} \frac{|\mathbb{X}_{st}|}{|t-s|^{2\alpha}} < \infty 
\] 
and satisfies the Chen's relation
\[
\mathbb{X}_{st} =\mathbb{X}_{su}+ \mathbb{X}_{ut} + X_{su} \otimes X_{ut} 
\] 
for all $0\leq s \leq u \leq t \leq T$.

We denote the set of $\alpha$-H\"older rough paths with respect to $X$ by $\mathcal{C}_X([0,T]; V)$.
\end{definition} 

Let  $\mathbf{X}=(X, \mathbb{X})$ and $\tilde{\mathbf{X}}=(\tilde{X}, \tilde{\mathbb{X}})$ be $\alpha$-H\"older rough paths. The $\alpha$-H\"older distance between $\mathbf{X}$ and $\tilde{\mathbf{X}}$ is given by
\[
\|\mathbf{X} ; \tilde{\mathbf{X}}\|_{\alpha}:=\|X-\tilde{X}\|_{\alpha} +\|\mathbb{X}-\tilde{\mathbb{X}}\|_{\alpha}.
\] 

\begin{definition}
Let $\mathbf{X}=(X, \mathbb{X})$ be an $\alpha$-rough path. The bracket of $\mathbf{X}$ is defined by
\[
[\mathbf{X}]_{st}=X_{st}\otimes X_{st}-2Sym (\mathbb{X}_{st})
\] 
where $Sym \mathbb{X}:= \frac{1}{2}(\mathbb{X}_{st} + \mathbb{X}_{st}^*)$ denote the symmetric part of $\mathbb{X}$.  Here $^*$ denotes the transpose, given by $(x\otimes y)^*=y \otimes x$ for all $x, y \in V$.

If $[\mathbf{X}]=0$ we say that $\mathbf{X}$ is a weak geometric rough path. 
\end{definition}

We observe that 
\[
\|X\|_{\infty} \leq |X_0|+ T^{\alpha} \| X\|_{\alpha} ~~;~~\| [\mathbf{X}]\|_{2\alpha} \leq \|X\|^2_{\alpha}+ 2\| \mathbb{X}\|_{2\alpha}.
\]
Let $\mathbf{X}=(X, \mathbb{X}) \in \mathcal{C}_X([0,T]; V)$ we have that there exists an canonical $\tilde{\mathbf{X}}=(X, \tilde{\mathbb{X})} \in \mathcal{C}_X([0,T]; V) $ weak geometric,  we called this rough path the geometrization of $\mathbf{X}$. It is defined by
\[
\tilde{\mathbb{X}}:= \mathbb{X} + \frac{1}{2}[\mathbf{X}].
\] 
Let 
\[
C^{\alpha}([0,T]; V)=\{ X:[0,T] \rightarrow V : \| X \|_{\alpha}<\infty \}
\]
denote the set of $\alpha$-H\"older functions from $[0,T]$ into $V$.

Let 
\[
C^{2\alpha}([0,T]^2; V)=\{ R:\triangle_T \rightarrow V : \| R \|_{2\alpha}:=\sup_{s,t} \frac{|R_{st}|}{|t-s|^{2\alpha}} < \infty \}
\]
denote the set of $2\alpha$-H\"older functions from $[0,T]^2$ into $V$.

We denote by $L(V,W)$ the space of continuous linear mappings from $V$ to $W$.  We have the following  natural operation from  $L(V, L(U,W))\times L(V,U)$ into $L(V\otimes V, U)$ given by
\[
(\psi \cdot \varphi) v\otimes v^{\prime}:=\psi (v)( \varphi(v^{\prime}))
\]
where $\psi \in L(V, L(U,W))$, $\varphi \in L(V, U)$ and $v, v^{\prime}\in V$.

For simplicity of notation, we identify $L(V,L(V,W))$ with $L(V\otimes V, W)$ via 
\[
\eta(x)(y)=\eta(x \otimes y)
\]
where $\eta \in L(V,L(V,W))$ and $x, y \in V$.
\begin{definition} Let  $X \in C^{\alpha}([0,T]; W)$. We say that a pair  $(Y, \partial_XY)$ is a controlled path with respect to $X$ if  $Y \in C^{\alpha}([0,T]; W)$,  $\partial_XY \in C^{\alpha}([0,T]; L(V,W))$ and $R^Y \in  C^{2\alpha}(\triangle_T; W)$ where
\[
R^Y_{st} :=Y_{st}-\partial_X Y_sX_{st} 
\]
for all $0\leq s \leq t \leq T$.
We said that $\partial_XY$ is a Gubinelli derivative of $Y$. 

We denoted by $\mathcal{D}_X^{2\alpha}=\mathcal{D}_X^{2\alpha}([0,T]; W)$ the space of controlled paths with respect to $X$.  
\end{definition}

We observe that $\mathcal{D}_X^{2\alpha}$ is a Banach space provided with the norm
\[
\| Y, \partial_XY\|_{\mathcal{D}_X^{2\alpha}}:= |Y_0|+ |\partial_XY_0|+ \|\partial_XY\|_{\alpha} + \| R^Y \|_{2\alpha}.
\]
To shorten notation, we write $\| Y\|_{\mathcal{D}_X^{2\alpha}}$ instead of $\| Y, \partial_XY\|_{\mathcal{D}_X^{2\alpha}}$. 

\vspace{.3in}

Following M. Gubinelli \cite{Gubinelli} we define the rough integral of a controlled path with respect to a rough path, see Theorem 4.10 of P. Friz and M. Hairer \cite{Friz Hairer}.

\begin{theorem}{\label{trp}}
Let $\mathbf{X}=(X, \mathbb{X}) \in \mathcal{C}^{\alpha}_X([0,T]; V)$, $(Y, \partial_XY) \in \mathcal{D}_X^{2\alpha}([0,T]; L(V,W))$. Then the limit
\[
\int_0^t Y_r d\mathbf{X}_r:= \lim_{|\pi| \rightarrow 0 }\sum_{[u,v]\in \pi}Y_uX_{uv} +\partial_XY_u \mathbb{X}_{uv}
\]
there exists for all $t \in [0.T]$, where the limit is taken over any sequence of partitions $\pi$ of the interval $[0,t]$  with mesh size $|\pi|$ convergent to zero. We call this limit the rough integral of the controlled path $(Y,\partial_XY)$ with respect to the rough path $\mathbf{X}$. Also holds the following estimative for all $s \leq t$,
\[
\vert \int_s^t Y_r d\mathbf{X}_r - Y_sX_{st}-\partial_XY_s \mathbb{X}_{st}\vert \leq C\lbrace    \|R^Y \|_{2\alpha} \|X\|_{\alpha} + \| \partial_XY\|_{\alpha} \|\mathbb{X}\|_{2\alpha}           \rbrace |t-s|^{3\alpha}
\]
where the constant $C$ depends only of $\alpha$.

\end{theorem}

We observe that holds the following conversion formula type It\^o-Stratonovich:
\begin{equation}\label{Ito-strato1}
\int_s^t Y_r d \tilde{\mathbf{X}}_r= \int_s^t Y_r d \mathbf{X}_r+ \frac{1}{2}\int_s^t \partial_XY_r d[ \mathbf{X}]_r
\end{equation}
for all $0 \leq s \leq t \leq T$. Here the integral with respect to $[\mathbf{X}]$ is understood in the Young sense, see M. Hairer and P. Friz  \cite{Friz Hairer} pp 61, and R. Castrequini and P. Catuogno \cite{Castrequini Catuogno}, M. Gubinelli \cite{Gubinelli}, P. Friz and N. Victoir \cite{Friz}, T. Lyons \cite{Lyons1} between others.

The next theorem establishes the integral of a controlled with respect to a controlled path, see \cite{Friz Hairer}. The proof is an easy consequence of the sewing lemma, see \cite{Friz Hairer} Lemma 4.2.

\begin{theorem}\label{teorema2}
Let $\mathbf{X}=(X, \mathbb{X}) \in \mathcal{C}^{\alpha}_X([0,T]; V)$. Let  $(Y, \partial_XY) \in \mathcal{D}_X^{2\alpha}([0,T]; L(W,U))$ and $(Z, \partial_XZ) \in \mathcal{D}_X^{2\alpha}([0,T]; W)$. Then the limit
\[
\int_0^t Y_r d_{\mathbf{X}}Z_r:= \lim_{|\pi| \rightarrow0 }\sum_{[u,v]\in \pi}Y_uZ_{uv} +\partial_XY_u \cdot \partial_XZ_u\mathbb{X}_{uv}
\]
there exists for all $t \in [0.T]$, where the limit is taken over any sequence of partitions $\pi$ of the interval $[0,t]$  with mesh size $|\pi|$ convergent to zero. We call this limit the integral of the controlled path $(Y,\partial_XY)$ with respect to the controlled path $(Z,\partial_XZ)$. Also holds the following estimative for all $s \leq t$,
\begin{eqnarray} 
\vert \int_s^t Y_r d_{\mathbf{X}}Z_r -Y_sZ_{st}-\partial_XY_s \cdot \partial_XZ_s \mathbb{X}_{st} \vert & \leq & C\lbrace    \|\partial_XY \|_{\infty} \|\partial_XZ\|_{\alpha} \|X\|^2_{\alpha} + \|Y\|_{\alpha}\|R^Z\|_{2\alpha} \nonumber \\ & & + \|R^Y\|_{2\alpha}\| \partial_XZ\|_{\infty} \|X\|_{\alpha}    +  \nonumber\\
& & + \| \partial_XY \cdot \partial_XZ\|_{\alpha} \|  \mathbb{X} \|_{2\alpha}       \rbrace |t-s|^{3\alpha} \nonumber
\end{eqnarray}
where the constant $C$ depends only of $\alpha$.
\end{theorem}
We observe that the integrals are controlled paths, in fact 
\begin{equation}\label{SF1}
(\int_0^{\cdot} Y_r d\mathbf{X}_r, Y) \in \mathcal{D}_X^{2\alpha}([0,T]; W) ~\mbox{ and }~(\int_0^{\cdot} Y_r d_{\mathbf{X}}Z_r, Y\partial_XZ) \in \mathcal{D}_X^{2\alpha}([0,T]; U).
\end{equation}




We also have the following conversion formula type It\^o-Stratonovich  for controlled paths
\begin{equation}\label{Ito-strato2}
\int_s^tY_r d_{\tilde{\mathbf{X}}}Z_r=\int_s^tY_r d_{\mathbf{X}}Z_r + \frac{1}{2}\int_s^t \partial_XY_r \cdot \partial_XZ_r d[\mathbf{X}]_r
\end{equation}
for all $0 \leq s \leq t \leq T$.  

\section{Chain rule for controlled paths} \label{3}
In this section we give a proof of the chain rule for controlled paths, it is a non trivial extension of Lemma 7.3 of P. Friz and M. Hairer \cite{Friz Hairer} and an improvement of the  Theorem 3.4 (i) of  C. Keller and J. Zhang  \cite{Keller Zhang}. 
\begin{lemma}\label{lemma}
Let $\alpha \in (\frac{1}{3}, \frac{1}{2}]$, $\mathbf{X}=(X, \mathbb{X}) \in \mathcal{C}^{\alpha}_X([0,T]; V)$ and $f : [0,T] \times U  \rightarrow W$ continuous and twice continuously differentiable in $U$. We assume that there exist continuous functions $ \partial_X f : [0,T]\times U \rightarrow L(V,W)$ and  $ \partial_X Df : [0,T]\times U \rightarrow L(V,L(U,W))$ such that
\begin{enumerate}
\item $(f, \partial_X f)\in C(U, \mathcal{D}_X^{2\alpha}([0,T]; W))$ and $(Df, \partial_X Df) \in C(U, \mathcal{D}_X^{2\alpha}([0,T]; L(U,W)))$,

\item  For each $t \in [0,T]$, $\partial_Xf(\cdot,t)\in C^1(U,L(V, W)$.
\end{enumerate}
Then for any $(Z, \partial_XZ) \in \mathcal{D}_X^{2\alpha}([0,T]; U)$ we have that  $t \rightarrow \eta_t=f(t,Z_t) \in  \mathcal{D}_X^{2\alpha}([0,T]; W)$ with
\begin{equation}
\partial_X\eta_t =\partial_Xf(t,Z_t)+ Df(t,Z_t)\partial_XZ_t.
\end{equation}
\end{lemma}
\begin{proof}

Let us first prove that $\partial_X\eta \in \mathcal{C}^ {\alpha}([0,T]; L(V,W))$. In fact,
\begin{eqnarray}
|\partial_X \eta_t-\partial_X \eta_s| & \leq & |\partial_X f(t,Z_t)-\partial_X f(s,Z_s)|  + |Df(t,Z_t) \partial_X Z_t-\nonumber \\& &-Df(s,Z_s)\partial f(s,Z_s)\partial_XZ_s|\nonumber \\
& \leq & \sup_{(s,t,r)\in [0,T]^ 2\times [0,1]} |D\partial_X f(t, Z_s +r Z_{st})||Z_{st}| +\sup_{s\in [0,T]}\| \partial_X f(\cdot, Z_s)\|_{\alpha} \nonumber \\   & &  |t-s|^{\alpha} +(\sup_{(s,t,r)\in [0,T]^ 2\times [0,1]} |D^2f(t, Z_s + rZ_{st})||Z_{st}| +\nonumber \\& & + \sup_{s \in [0,T]}|   Df(\cdot, Z_s)||t-s|^{\alpha} )| \partial_XZ_s| + \nonumber \\& &  + \sup_{t \in [0,T]}\| Df(t,Z_t)\|_{\infty}\| \partial_X Z\|_{\alpha}|t-s|^{\alpha}.
\end{eqnarray}
Thus 
\begin{equation}
|\partial_X \eta_t-\partial_X \eta_s| \leq K_{f, Z}|t-s|^{\alpha} 
\end{equation} 
where $K_{f,Z}$ is a finite constant given by
\begin{eqnarray}
K_{f,Z} &=& \| Z\|_{\alpha}(\sup_{(s,t,r)\in [0,T]^ 2\times [0,1]} |D\partial_X f(t, Z_s +r Z_{st})| +\sup_{(s,t,r)\in [0,T]^ 2\times [0,1]} |D^2f(t, Z_s + rZ_{st})|) + \nonumber \\ & & +\sup_{s\in [0,T]}\| \partial_X f(\cdot, Z_s)\|_{\alpha} +  \sup_{s \in [0,T]}\|   Df(\cdot, Z_s)\|_{\alpha}\| \partial_X Z\|_{\infty}+ \nonumber \\
& & + \sup_{t \in [0,T]}\|   Df(t, Z_t)\|_{\infty}\| \partial_X Z\|_{\alpha}.
\end{eqnarray}

We proceed to show that $\eta_{st}- \partial_X\eta_s X_{st} \in C^ {2\alpha}([0,T]; W)$. 

Let $0 \leq s \leq t \leq T$. Then

\begin{eqnarray*}
\eta_{st} & =  & f(t,Z_t)-f(s,Z_s) \\
& = &  (f(t, Z_t) -f(t,Z_s))+(f(t, Z_s) -f(s,Z_s)) \\
& = & \partial_Xf(s,Z_t)X_{st}+ R_{st}^{f(\cdot, Z_t)}+ \int_0^1 Df(s, Z_s+ rZ_{st})\cdot Z_{st}dr \\
& = & \left(\partial_Xf(s,Z_s)+ Df(s,Z_s)\partial_XZ_s\right)X_{st}+ R^{\eta}_{st} 
\end{eqnarray*}
where
\begin{eqnarray}\label{eq1}
R^{\eta}_{st} & = & (\partial_Xf(s,Z_t)-\partial_Xf(s,Z_s))X_{st}+ R_{st}^{f(\cdot, Z_t)}+ \nonumber\\
&  & + \int_0^1 (Df(s, Z_s + rZ_{st})-Df(s, Z_s)) dr \cdot \partial_XZ_sX_{st}+ \nonumber \\
& & +  \int_0^1 Df(s, Z_s+ r Z_{st})) dr  \cdot R^Z_{st}. 
\end{eqnarray}

Now we compute the following estimatives
\begin{eqnarray}\label{eq2}
|(\partial_Xf(s,Z_t)-\partial_Xf(s,Z_s))X_{st}|  & \leq & |\partial_Xf(s,Z_t)-\partial_Xf(s,Z_s)| |X_{st}| \nonumber\\
& \leq & \sup_{(s,t,r) \in [0,T]^2\times [0,1]}|D\partial_Xf(s,Z_s +rZ_{st})||Z_{st}||X_{st}| \nonumber \\
& = & B_{f,Z} \|Z\|_{\alpha}\|X\|_{\alpha}|t-s|^ {2\alpha},
\end{eqnarray}
\begin{eqnarray}\label{eq3}
|R_{st}^{f(\cdot, Z_t)}| & \leq & \| R_{st}^{f(\cdot, Z_t)} \|_{2\alpha}|t-s|^ {2\alpha} \nonumber \\
& \leq &\sup_{t \in [0,T]} \| f(\cdot, Z_t), \partial_X f(\cdot, Z_t )   \|_{2\alpha}|t-s|^{2\alpha} \nonumber \\
& = & C_{f,Z}|t-s|^{2\alpha},
\end{eqnarray}

\begin{eqnarray}\label{eq4}
| \int_0^1 (Df(s, Z_s + rZ_{st})-Df(s, Z_s)) dr \cdot \partial_XZ_sX_{st} | & \leq & | \int_0^ 1\int_0^1D^2f(s,Z_s+r^{\prime} r Z_{st}) \nonumber \\ &  & Z_{st}r^{\prime}\partial_XZ_s X_{st}dr^{\prime}dr | \nonumber \\
& \leq & \sup_{(s,t,r,r^{\prime})\in [0,T]^2 \times [0,1]^2}|D^2f(s, Z_s+r^{\prime} r Z_{st})| \nonumber \\ 
& & \|\partial_XZ\|_{\infty}\|X\|_{\alpha} \|Z\|_{\alpha}|t-s|^ {2\alpha} \nonumber \\
& = & D_{f,Z}\|\partial_XZ\|_{\infty}\|X\|_{\alpha} \|Z\|_{\alpha}|t-s|^ {2\alpha}
\end{eqnarray}
and 
\begin{eqnarray}\label{eq5}
| \int_0^1 Df(s, Z_s+ r Z_{st})) dr  \cdot R^Z_{st} | & \leq & \sup_{(s,t,r) \in [0,T]^2\times [0,1]}|Df(s,Z_s +rZ_{st})| |R^Z_{st}| \nonumber \\
& \leq & E_{f, Z}\|R^Z\|_{\alpha}|t-s|^{2\alpha} 
\end{eqnarray}
where
\[
B_{f,Z}:=\sup_{(s,t,r) \in [0,T]^2\times [0,1]}|D\partial_Xf(s,Z_s +rZ_{st})| ~ ; ~ C_{f,Z}:=\sup_{t \in [0,T]} \| f(\cdot, Z_t), \partial_X f(\cdot, Z_t )   \|_{2\alpha}
\]

\[
 D_{f,Z}:=\sup_{(s,t,r,r^{\prime})\in [0,T]^2 \times [0,1]^2}|D^2f(s, Z_s+r^{\prime} r Z_{st})| ~ ; ~ E_{f, Z}:=\sup_{(s,t,r) \in [0,T]^2\times [0,1]}|Df(s,Z_s +rZ_{st})| 
\]
are finite from continuity of the functions and compactness of the sets envolved. 

Combining (\ref{eq1}), (\ref{eq2}), (\ref{eq3}), (\ref{eq4}) and (\ref{eq5}) we obtain 
\[
| R^{\eta}_{st} | \leq (B_{f,Z} \|Z\|_{\alpha}\|X\|_{\alpha} +C_{f,Z} +D_{f,Z}\|\partial_XZ\|_{\infty}\|X\|_{\alpha} \|Z\|_{\alpha} +  E_{f, Z}\|R^Z\|_{\alpha} )|t-s|^{2 \alpha}
\]
Thus $R^{\eta} \in C^{2\alpha}([0,T]; W)$ and the proof is complete.
\end{proof}

\begin{corollary}
Let $\alpha \in (\frac{1}{3}, \frac{1}{2}]$, $\mathbf{X}=(X, \mathbb{X}) \in \mathcal{C}^{\alpha}_X([0,T], V)$,  $f : U  \rightarrow W$ twice continuously differentiable in $U$ and $(Z, \partial_XZ) \in \mathcal{D}_X^{2\alpha}([0,T]; U)$. Then  we have that  $t \rightarrow \eta_t=f(Z_t) \in  \mathcal{D}_X^{2\alpha}([0,T]; U)$ with
\[
\partial_X\eta_t =Df(Z_t)\partial_XZ_t.
\]
\end{corollary}

\section{The It\^o-Wentzell formula} \label{4}

The It\^o-Wentzell formula for rough paths is established by our next theorem.

\begin{theorem}\label{Ito}
Let $\alpha \in (\frac{1}{3}, \frac{1}{2}]$,  $\mathbf{X}=(X, \mathbb{X}) \in \mathcal{C}^{\alpha}_X([0,T]; V)$ and $(h, \partial_Xh)  \in C(U, \mathcal{D}_X^{2\alpha}([0,T],L(V,W)))$. We assume that   
\begin{enumerate}
\item $h: [0,T]\times U \rightarrow L(V,W)$ is continuous and twice differentiable in $U$.
\item $Dh: [0,T]\times U \rightarrow L(U,L(V,W))$ is continuous and differentiable in $U$. 
\item $(Dh, \partial_X Dh) \in C(U, \mathcal{D}_X^{2\alpha}([0,T],L(U,L(V,W))))$ 
\end{enumerate}
Let
\[
 g(t,x) = g(0,x) + \int_0^t h(s,x)d\mathbf{X}_s.
\]
Assume that $g: [0,T] \times U \rightarrow W$ is twice differentiable in $U$ and the functions  $(t, x) \mapsto  D g_t(x)$  and $(t, x) \mapsto  D^2 g_t(x)$ are continuous.
Then for any $Z \in \mathcal{D}_X^{\alpha}([0,T];U)$,
\begin{eqnarray}\label{IV0}
 g(t,Z_t) & = & g(0,Z_0) + \int_0^t h(r,Z_r)d\mathbf{X}_r + \int_0^t Dg(r,Z_r) d_{\mathbf{X}}Z_r+\nonumber \\
 & &  + \int_0^ tDh(r,Z_r)\circ \partial_X Z_rd[\mathbf{X}]_{r} + \nonumber \\ & & +\frac{1}{2}\int_0^tD^ 2g(r, Z_r)  \partial_XZ_r \otimes \partial_X Z_r d[ \mathbf{X} ]_{r}.
\end{eqnarray}

\end{theorem}
\begin{proof}
Let $0 \leq s \leq t \leq T$. Our proof starts with the following observation
\begin{equation}\label{f}
g(t,Z_t)-g(s,Z_s)= (g(t, Z_t) -g(t,Z_s))+(g(t, Z_s) -g(s,Z_s)).
\end{equation}

According to Lemma \ref{lemma} we have that  $r \rightarrow Dg(r, Z_r)$ and $r \rightarrow h(r,Z_r)=\partial_Xg(r, Z_r)$ are controlled paths with Gubinelli derivatives,
\begin{eqnarray}
\partial_X (Dg(s,Z_s))& = & (\partial_X Dg)(s, Z_s)+ D^2g(s,Z_s)\partial_XZ_s \label{G1}\\
\partial_X (h(s,Z_s))& = & \partial_X h(s, Z_s)+ Dh(s,Z_s)\partial_XZ_s. \label{G2}
\end{eqnarray}
It is clear from the definitions  that 
\begin{equation}
\partial_X (Dg(s,Z_s)) \cdot \partial_XZ_s
=  (\partial_X Dg)(s, Z_s)\cdot \partial_XZ_s+ D^2g(s,Z_s)\partial_XZ_s \otimes \partial_XZ_s. \label{G3}\
\end{equation}

We begin  studying  $g(t, Z_t) -g(t,Z_s)$ the first term of (\ref{f}). By Taylor's formula, 
\begin{equation}\label{t}
g(t, Z_t) -g(t,Z_s)=Dg(t, Z_s)Z_{st}+ \frac{1}{2}D^2g(t,Z_s)Z_{st} \otimes Z_{st} + R^0_{st}
\end{equation}
where $R^0_{st}$ is the Taylor's remainder. It is clear that $|R^0_{st}|\leq C|t-s|^{3\alpha}$. 

We claim that 
\begin{eqnarray}\label{a1} 
\frac{1}{2}D^2g(t,Z_s)Z_{st}\otimes Z_{st} &=& D^2g(s,Z_s)(\partial_XZ_s \otimes \partial_X Z_s )\mathbb{X}_{st}+ \nonumber \\ & & +\frac{1}{2}D^ 2g(s, Z_s)\partial_XZ_s \otimes \partial_X Z_s [ \mathbf{X} ]_{st}+ R^1_{st}
\end{eqnarray}
and 
\begin{equation}\label{a2}
Dg(t,Z_s)Z_{st} = Dg(s,Z_s)Z_{st} +(Dh(s,Z_s) \circ \partial_X Z_s )X_{st} \otimes X_{st} +R^2_{st}
\end{equation}
where $|R^1_{st} |, |R^2_{st}| \leq C|t-s|^{3\alpha}$. For the proofs see Appendix \ref{A1} and Appendix \ref{A2}.

Substituting (\ref{a1}) and (\ref{a2}) into (\ref{t}) we obtain
\begin{eqnarray}\label{1taylor}
g(t, Z_t) -g(t,Z_s) & = & Dg(s,Z_s)Z_{st}+ (Dh(s,Z_s) \circ \partial_X Z_s )X_{st} \otimes X_{st}  + \nonumber \\
& & + D^2g(s,Z_s)(\partial_XZ_s \otimes \partial_X Z_s ) \mathbb{X}_{st}+ \nonumber \\ & & 
+\frac{1}{2}D^ 2g(s, Z_s)\partial_XZ_s \otimes \partial_X Z_s [ \mathbf{X} ]_{st}+R^3_{st}
\end{eqnarray}
where $R^3_{st}=R^0_{st}+R^1_{st}+R^2_{st}$. Obviously, $|R^3_{st}|\leq C|t-s|^{3\alpha}$. 

The estimative of Theorem \ref{teorema2} and (\ref{G3}) yields
\begin{eqnarray}\label{int}
\int_s^tDg(r, Z_r)d_{\mathbf{X}}Z_r &=&Dg(s, Z_s)Z_{st} + \partial_X(Dg(s,Z_s))\cdot \partial_X Z_s \mathbb{X}_{st} + R^4_{st} \nonumber \\
& = &  Dg(s, Z_s)Z_{st} + ((\partial_X Dg)(s, Z_s)\cdot \partial_XZ_s+ \nonumber \\ 
& & + D^2g(s,Z_s)(\partial_XZ_s \otimes \partial_XZ_s) \mathbb{X}_{st}   + R^4_{st}
\end{eqnarray}
where $|R^4_{st}|\leq C|t-s|^{3\alpha}$.

Substituting (\ref{int}) into (\ref{1taylor}) we obtain 
\begin{eqnarray}
g(t, Z_t) -g(t,Z_s)& =& \int_s^tDg(r, Z_r)d_{\mathbf{X}}Z_r  + (Dh(s,Z_s) \circ \partial_X Z_s )X_{st} \otimes X_{st}  - \nonumber \\ &  & - ((\partial_X Dg)(s, Z_s)\cdot \partial_XZ_s)\mathbb{X}_{st}  + \nonumber \\ & & \label{ult}
+\frac{1}{2}D^ 2g(s, Z_s)\partial_XZ_s \otimes \partial_X Z_s [ \mathbf{X} ]_{st} + R^5_{st} 
\end{eqnarray}
where  $R^ 5 = R^4 -R^ 3$. 

The  commutativity between $D$ and $\partial_X$  implies that
\begin{equation}\label{com}
((\partial_X Dg)(s, Z_s)\cdot \partial_XZ_s)\mathbb{X}_{st}  =  ((D \partial_X g)(s, Z_s)\circ  \partial_XZ_s)\mathbb{X}_{st}^*,
\end{equation}
see Appendix \ref{A3} for a proof.

Combining (\ref{ult}), (\ref{com}) and   $h=\partial_X g$ we have that
\begin{eqnarray}
g(t, Z_t) -g(t,Z_s)& =& \int_s^tDg(r, Z_r)d_{\mathbf{X}}Z_r  + (Dh(s,Z_s) \circ \partial_X Z_s ) \nonumber  \\ & & X_{st} \otimes X_{st}  - (Dh(s, Z_s)\circ \partial_XZ_s)\mathbb{X}^*_{st}  + \nonumber \\ 
& & + \frac{1}{2}D^ 2g(s, Z_s)\partial_XZ_s \otimes \partial_X Z_s [ \mathbf{X} ]_{st} + R^5_{st} \label{ult1}
\end{eqnarray}

Now, we proceed to study $g(t, Z_s)-g(s,Z_s)$ the second term of (\ref{f}).  By definitions and Theorem \ref{trp},
\begin{eqnarray}
g(t, Z_s)-g(s,Z_s) & = & \int_s^t h(r, Z_s)d\mathbf{X}_r \nonumber \\
& = & h(s, Z_s)X_{st} + \partial_Xh (s,Z_s) \mathbb{X}_{st} + R^6_{st} \label{l1}
\end{eqnarray}
where $|R^6_{st}| \leq |t-s|^{3\alpha}$.

By Lemma \ref{lemma} and Theorem  \ref{trp},

\begin{eqnarray}
\int_s^t h(r, Z_r)d\mathbf{X}_r & = & h(s, Z_s)X_{st}  + \partial_X(h (s,Z_s)) \mathbb{X}_{st} + R^7_{st}\nonumber \\
& = & h(s, Z_s)X_{st} + (\partial_Xh (s,Z_s) + Dh(s, Z_s)\circ \partial_XZ_s )\mathbb{X}_{st} +\nonumber \\
& & + R^7_{st} \label{l2}
\end{eqnarray}
where $|R^7_{st}| \leq |t-s|^{3\alpha}$.

Substituting (\ref{l2}) into (\ref{l1}) yields
\begin{eqnarray}
g(t, Z_s)-g(s,Z_s) & = & \int_s^t h(r, Z_r)d\mathbf{X}_r -(Dh(s,Z_s)\circ \partial_X Z_s)\mathbb{X}_{st} + \nonumber \\
&  & + R^8_{st} \label{l3}
\end{eqnarray}
where $R_{st}^8:=R_{st}^6-R_{st}^7$. Obviously, $|R^8_{st}| \leq |t-s|^{3\alpha}$.

Finally, substituting (\ref{l3}) and (\ref{ult1}) in (\ref{f}) and recalling that 
\[
[\mathbf{X}]_{st}=X_{st}\otimes X_{st}- 2 Sym(\mathbb{X}_{st}),
\] 
we obtain that
 
\begin{eqnarray}\label{f1}
g(t,Z_t)-g(s,Z_s) &=& (g(t, Z_t) -g(t,Z_s))+(g(t, Z_s) -g(s,Z_s)) \nonumber \\
& = &  \int_s^tDg(r, Z_r)d_{\mathbf{X}}Z_r  + (Dh(s,Z_s) \circ \partial_X Z_s) X_{st} \otimes X_{st} - \nonumber \\ &  & - (Dh(s, Z_s)\circ \partial_XZ_s)\mathbb{X}^*_{st} + \frac{1}{2}D^ 2g(s, Z_s)\partial_XZ_s \otimes \partial_X Z_s [ \mathbf{X} ]_{st} \nonumber \\ 
& & +\int_s^t h(r, Z_r)d\mathbf{X}_r   -(Dh(s,Z_s)\circ \partial_X Z_s)\mathbb{X}_{st} +R^9_{st}  \nonumber \\
& = & \int_s^tDg(r, Z_r)d_{\mathbf{X}}Z_r + \int_s^t h(r, Z_r)d\mathbf{X}_r + \nonumber \\
& & +(Dh(s,Z_s)\circ \partial_X Z_s)[\mathbf{X}]_{st} + \frac{1}{2}D^ 2g(s, Z_s)\partial_XZ_s \otimes \partial_X Z_s [ \mathbf{X} ]_{st} \nonumber \\ 
& = &  + R^9_{st}
\end{eqnarray}
where $R^9_{st} =R^5_{st} + R^8_{st}$. Obviously  $R^9_{st} \leq C|t-s|^{3\alpha}$

Since $3\alpha >1$ and $R^9_{st} \leq C|t-s|^{3\alpha}$ we have that  
\[
\lim_{|\pi| \rightarrow 0 }\sum_{[u,v]\in \pi}R^9_{uv}=0
\]
where the limit is taken over any sequence of partitions $\pi$ of the interval $[0,t]$  with mesh size $|\pi|$ convergent to zero. 

By other hand,  
\[
\lim_{|\pi| \rightarrow 0 }\sum_{[u,v]\in \pi}(Dh(u,Z_u)\circ \partial_X Z_u)[\mathbf{X}]_{uv} =\int_0^t(Dh(r,Z_r)\circ \partial_X Z_r)d[\mathbf{X}]_r
\]
and 
\[
\lim_{|\pi| \rightarrow 0 }\sum_{[u,v]\in \pi}D^ 2g(u, Z_u)\partial_XZ_u \otimes \partial_X Z_u [ \mathbf{X} ]_{uv} =\int_0^tD^ 2g(r, Z_r)\partial_XZ_r \otimes \partial_X Z_r d[ \mathbf{X} ]_r \
\]
where the integrals are taken in the sense of Young. 
Thus 
\begin{eqnarray}
g(t,Z_t)& = & g(0,Z_0)+\int_0^tDg(r, Z_r)d_{\mathbf{X}}Z_r + \int_0^t h(r, Z_r)d\mathbf{X}_r + \nonumber \\
& & + \int_0^ tDh(r,Z_r)\circ \partial_X Z_r)d[\mathbf{X}]_{r} + \nonumber \\ & & + \frac{1}{2}\int_0^tD^ 2g(r, Z_r)  \partial_XZ_r \otimes \partial_X Z_r d[ \mathbf{X} ]_{r} 
\end{eqnarray}
for all $t \in [0,T]$.

\end{proof}
  
\begin{remark}
We observe that in the hypotesis of the above theorem with $\mathbf{X}=(X, \mathbb{X}) \in \mathcal{C}^{\alpha}_X([0,T]; V)$ weak geometric, holds the following formula  
\begin{equation}\label{Strato}
g(t,Z_t)  =  g(0,Z_0) + \int_0^t h(s,Z_s)d\mathbf{X}_s + \int_0^t Dg(s,Z_s) d_{\mathbf{X}}Z_s.
\end{equation}
This formula is a rough path version of the It\^o -Wentzell formula for the Stratonovich integral, see J-M. Bismut\cite{Bismut} and H. Kunita \cite{Kunita3} and \cite{Kunita2}. 

Reciprocally, we have that  (\ref{IV0}) is a consequence of (\ref{Strato}). In fact, we have that (\ref{Strato}) holds for $\tilde{\mathbf{X}}$ the geometrization of $\mathbf{X}$. Applying (\ref{Ito-strato1}) and (\ref{Ito-strato2}) into (\ref{Strato}),
\begin{eqnarray}\label{st}
g(t,Z_t)  & = & g(0,Z_0) + \int_0^t h(s,Z_s)d\mathbf{\tilde{X}}_s + \int_0^t Dg(s,Z_s) d_{\mathbf{\tilde{X}}}Z_s  \nonumber \\
 & = & (   \int_0^t h(s,Z_s)d\mathbf{X}_s   + \frac{1}{2}\int_0^ tDh(r,Z_r)\circ \partial_X Z_r)d[\mathbf{X}]_{r} ) + \nonumber \\
 & & (  \int_0^t Dg(s,Z_s) d_{\mathbf{X}}Z_s+\frac{1}{2}\int_0^t\partial_X(Dg(s,Z_s))\cdot \partial_X Z_r d[ \mathbf{X} ]_{r}   ). \nonumber \\
 & & 
\end{eqnarray}
By  Lemma 3.1 we have 
\begin{eqnarray}
\partial_X (Dg(s,Z_s)) \cdot \partial_XZ_s
& = & (\partial_X Dg)(s, Z_s)\cdot \partial_XZ_s+ D^2g(s,Z_s)\partial_XZ_s \otimes \partial_XZ_s\nonumber \\ & & \label{oooo}
\end{eqnarray}
and it follows easily from (\ref{com}) that
\begin{equation}\label{remark}
\int_0^t(\partial_X Dg)(r, Z_r)\cdot \partial_XZ_rd[\mathbf{X}]_{r}  =  \int_0^tD h(r, Z_r)\circ  \partial_XZ_rd[\mathbf{X}]_{r}.
\end{equation}
Substituting (\ref{remark}) and (\ref{oooo}) into (\ref{st}) we obtain (\ref{IV0}).
\end{remark}
\begin{remark}
We observe that the above theorem give the It\^o-Wentzell formula of C. Keller and J. Zhang (Theorem 3.4 of \cite{Keller Zhang}). In fact, let $g(t,x)$ in the hypotesis of Theorem \ref{Ito} and $Z \in \mathcal{D}_X^{\alpha}([0,T];U)$ given by
\[
Z_t = Z_0 + \int_0^t a_r d \mathbf{X}_r + \int_0^t b_r d[\mathbf{X}]_r
\]
where $(a, \partial_X a) \in \mathcal{D}_X^{\alpha}([0,T];L(V,U))$ and $b \in C^{\alpha}([0,T], L(V \otimes V , U))$. 

It is easy to check that the following substitution rule holds
\begin{equation}\label{SR4}
\int_0^ t c_rd_{\mathbf{X}}Z_r= \int_0^tc_ra_rd\mathbf{X}_r+\int_0^tc_r b_r d[\mathbf{X}]_r
\end{equation}
for all $(c, \partial_X c) \in \mathcal{D}_X^{\alpha}([0,T];L(U,W))$.

Thus applying  (\ref{IV0}) and recalling that $\partial_XZ=a$, we have that
\begin{eqnarray}
g(t,Z_t)& = & g(0,Z_0)+  \int_0^t h(x, Z_r) d \mathbf{X}_r +\int_0^t Dg(r, Z_r) d_{ \mathbf{X}} Z_r  + \nonumber \\
&  &+ \int_0^t Dh(r,Z_r)\circ \partial_X Z_rd [\mathbf{X}]_r + \frac{1}{2} \int_0^t D^2g(r, Z_r) \partial_XZ_r \otimes \partial_X Z_rd [\mathbf{X}]_r \nonumber \\
& = & g(0,Z_0)+  \int_0^t h(x, Z_r) d \mathbf{X}_r +(\int_0^t Dg(r, Z_r) a_r d\mathbf{X}_r +\int_0^t Dg(r, Z_r)  b_r d [\mathbf{X}]_r)  + \nonumber \\
&  &+ \int_0^t Dh(r,Z_r)\circ a_rd [\mathbf{X}]_r + \frac{1}{2} \int_0^t D^2g(r, Z_r) a_r \otimes a_rd [\mathbf{X}]_r \nonumber \\
& = & g(0,Z_0)+  \int_0^t (h(x, Z_r) +Dg(r, Z_r) a_r )  d \mathbf{X}_r + \nonumber \\ 
& & + \int_0^t  (  Dg(r, Z_r)b_r + Dh(r,Z_r)\circ a_r + \frac{1}{2} D^2g(r, Z_r) a_r \otimes a_r )d[ \mathbf{X}]_r. \nonumber \\
& &
\end{eqnarray}

\end{remark}
{\bf Composition of flows:} The following result is an extension of the H. Kunita formula for the composition of flows defined by
stochastic differential equations, see \cite{Kunita1}.
\begin{corollary}
 Let $Y$ and $Z$ be the flows associated with the rough differential equations 
\[
dY = g(Y) d\mathbf{X}~ \mbox{ and }~ dZ =f(Z)d\mathbf{X}
\] 
respectively. 
Then $V= Y \circ Z$ satisfies
\begin{eqnarray*}
 dV = (g(V) + Y_*f(V))d\mathbf{X} +(Dg(Z)\circ f(Z)+  D^2Y(Z)f(Z)\otimes f(Z))d[\mathbf{X}]
\end{eqnarray*}
In case that $\mathbf{X}$ is weak geometric, 
\[
 dV = g(V)d\mathbf{X} + Y_*f(V)d\mathbf{X}.
\]

\end{corollary}
\begin{proof}
By assumption,
\[
Y_t(x) = x + \int_0^t g(Y_s(x))d\mathbf{X}_s.  
\]
Applying the It\^o-Wentzell formula, the substitution rule and the expresion for the Gubinelli derivative of the solution of a rough differential equation we obtain
\begin{eqnarray*}
V_t & = &  x + \int_0^t g(Y_r(Z_r))d\mathbf{X}_s + \int_0^t D_xY_r(Z_r)d_{\mathbf{X}}Z_r + \\
& & \int_0^ tDg(Z_r)\circ \partial_XZ_rd[\mathbf{X}]_r+\int_0^ tD^2 Y_r(Z_r) \partial_XZ_r\otimes \partial_X Z_rd[\mathbf{X}]_r \\ & =& x + \int_0^tg(V_r)d\mathbf{X}_r + \int_0^t D_xY_r(Z_r)f(Z_r)d\mathbf{X}_r +\\ 
&  & \int_0^t Dg(Z_r)\circ f(Z_r) d[\mathbf{X}]_r + \int_0^t D^2Y_r(Z_r)f(Z_r)\otimes f(Z_r)d[\mathbf{X}]_r \\
& = & x + \int_0^t(g(V_r)+ (Y_r)_* f(V_r))d\mathbf{X}_r+\\
& & \int_0^t (Dg(Z_r)\circ f(Z_r)+  D^2Y_r(Z_r)f(Z_r)\otimes f(Z_r))d[\mathbf{X}]_r .
\end{eqnarray*}
\end{proof}

{\bf Rough partial differential equations:}  Let $ \mathbf{X}=(X, \mathbb{X}) \in \mathcal{C}^{\alpha}([0,T];\mathbf{R}^n)$ be a weak geometric rough path with  $\alpha\in(\frac{1}{3},\frac{1}{2}]$,  $\phi : \mathbf{R}^d \rightarrow \mathbf{R}$ and $F :  \mathbf{R}^d \times \mathbf{R} \times \mathbf{R}^d \rightarrow L(\mathbf{R}^n,\mathbf{R})$ given by $F(x,u,v)(w_1,\cdots, w_n )=\sum_{i=1}^ nF^j(x,u,v)w_j$. We will consider the following first order rough differential equation

\begin{equation}\label{transport}
 \left \{
\begin{array}{lll}
    du(t,x) = F ( x, u(t,x) ,D_x u(t,x) )d\mathbf{X}_t\\
 u(0,x) = \phi(x).
\end{array}
\right .
\end{equation}
We assume that   
$F^j  \in C^{3}(\mathbf{R}^{2d+1})$ for $j =1 , \cdots , n$.
  
\begin{definition}
Let $\phi \in C^1( \mathbf{R}^d)$. We say that a function $u: [0,T] \times \mathbf{R}^d \rightarrow \mathbf{R}$  is a solution of  (\ref{transport}) with the initial condition $u_0 = \phi$ if 
\begin{enumerate}
\item $u \in C(\mathbf{R}^d, \mathcal{D}_X^{2\alpha }([0,T], \mathbf{R}))$ and $Du \in C(\mathbf{R}^d, \mathcal{D}_X^{2\alpha }([0,T],\mathcal{L}(\mathbf{R}^d,\mathbf{R}))$.
\item $(t,x)\rightarrow D^2u(t,x) $ is continuous.
\item For all $(t,x)\in [0,T]\times \mathbf{R}^d$, 
\[
 u(t,x) = \phi(x) +\int_0^t F ( x , u(r,x), D_x u(r,x)) d\mathbf{X}_r.
\]
\end{enumerate}  
\end{definition}
We denote $F_u ^j=D_uF^j$, $F_{x_i}^j= D_{x_i}F^j$, $F_{p_i}^j= D_{p_i}F^j$,  $F_x ^j=(F_{x_1}^j, \cdots , F_{x_d}^j)$ and $F_p^j =(F_{p_1}^j, \cdots , F_{p_d}^j)$.

The characteristic system associated with (\ref{transport}) is defined by the rough differential equation
\[
Y_t(y)=y + \int_0^t f(Y_s(y))d\mathbf{X}_s
\]
where $f : \mathbf{R}^d\times \mathbf{R} \times \mathbf{R}^d \rightarrow L(\mathbf{R}^n;\mathbf{R}^d\times \mathbf{R} \times \mathbf{R}^d) $ is given by
\begin{eqnarray*}
f(x,u,p)(v)& = & (-\sum_{j=1}^n F_p^j(x,u,p)v_j, \sum_{j=1}^n (F^j(x,u,p) -F^j_p(x,u,p) \cdot p)v_j,   \\ & & \sum_{j=1}^n( F_x^j(x,u,p)+ F^j_u(x,u,p)p) v_j  ).
\end{eqnarray*}

We observe that the characteristic system associated with (\ref{transport}) is written as
\begin{equation}\label{characteristic}
 \left \{
\begin{array}{lll}
    da_t & = & -\sum_{j=1}^n F_p^j (a_t, b_t ,c_t )d{\mathbf{X}}^j_t\\
 db_t & = & \sum_{j=1}^n \{ F^j (a_t, b_t ,c_t )-F_p^j (a_t, b_t ,c_t ) \cdot c_t \} d{\mathbf{X}}^j_t \\
dc_t & = & \sum_{j=1}^n \{ F_x^j (a_t, b_t ,c_t ) + F_u^j (a_t, b_t ,c_t ) c_t \} d{\mathbf{X}}^j_t
\end{array}
\right.
\end{equation}
where  $Y_t(y)=(a_t(y), b_t(y),c_t(y)) \in  \mathbf{R}^d\times \mathbf{R} \times \mathbf{R}^d $.
 
The structure of characteristics for equation (\ref{transport}), follows in a similar way that Theorem 4.5 of \cite{Castrequini Catuogno}. This is, if $u$ is a solution of (\ref{transport}) and $a_t$ solves the equation
\begin{equation}\label{E1}
 da_t  =  -\sum_{j=1}^n F_p^j (a_t, b_t ,c_t )d{\mathbf{X}}^j_t
\end{equation}
where $b_t := u(t, a_t)$ and $c_t := D_xu (t, a_t)$. Then we have that $(a_t, b_t , c_t)$ solves the characteristic system (\ref{characteristic}).

The existence and uniqueness of solution for a rough semilinear equation is given by the method of characteristics.

\begin{theorem}
The rough semilinear differential equation 
\begin{equation}\label{semilinear}
 \left \{
\begin{array}{lll}
    du(t,x) = \sum_{j=1}^n (D_xu(t,x)(P^j (x)) +Q^j(x, u(t,x)))d \mathbf{X}^j_t\\
 u(0,x) = \phi(x).
\end{array}
\right .
\end{equation}
has a unique solution given by
\[
u(t,x)=\tilde{b}_t\circ \tilde{a}_t^{-1}(x)
\]
where $\tilde{a}_t(x)= a_t(x, \phi(x),D\phi(x))$ and $\tilde{b}_t(x)= b_t(x, \phi(x),D\phi(x))$.
\end{theorem}

\begin{proof}
We observe that $F^j(x,u,p)= P^j(x)\cdot p +Q^j(x,u)$, then
\[
F^j_p(x,u,p)= P^j(x) ~; ~F^j(x,u,p)-F^j_p(x,u,p) \cdot p=Q^j(x,u)  
\]
 for $j=1, \cdots , n$.

Thus, regarding the projection $(a_t, b_t)$ of $(a_t,b_t, c_t)$  the unique solution of  (\ref{characteristic}) satisfies,
\begin{equation}\label{characteristic1}
 \left \{
\begin{array}{lll}
    da_t & = & -\sum_{j=1}^n P^j (a_t )d\mathbf{X}^j_t\\
 db_t & = & \sum_{j=1}^n Q^j (a_t, b_t ) d\mathbf{X}^j_t 
\end{array}
\right .
\end{equation}
We observe that 
\begin{equation}
da^{-1}_r(x) =D_x\tilde{a}^{-1}_r(x)\sum_{j=1}^n P^j (x)d\mathbf{X}^j_r 
\end{equation}

By It\^o-Wentzell formula, 
\begin{eqnarray*}
\tilde{b}_t \circ \tilde{a}_t^{-1}(x) & = & \phi(x)+ \int_0^t  \sum_{j=1}^nQ^j (x, \tilde{b}_r(\tilde{a}_r^{-1}(x))d\mathbf{X}^j_r\\
& & + \int_0^t D_x\tilde{b}_r(\tilde{a}_r^{-1}(x)da^{-1}_r(x) \\
& = &  \phi(x)+ \sum_{j=1}^n  \int_0^t Q^j (x, \tilde{b}_r(\tilde{a}_r^{-1}(x))d\mathbf{X}^j_r \\
& & + \int_0^t D_x\tilde{b}_r(\tilde{a}_r^{-1}(x))D_x\tilde{a}^{-1}_r(x)\sum_{j=1}^n P^j (x)d\mathbf{X}^j_r\\
& = & \phi(x)+ \sum_{j=1}^n \int_0^t   (Q^j (x, \tilde{b}_r \circ \tilde{a}_r^{-1}(x) \\
& & + D_x\tilde{b}_r\circ \tilde{a}_r^{-1}(x) P^j (x))d\mathbf{X}^j_r.
\end{eqnarray*}
\end{proof}

\section{Appendix} \label{A}

\subsection{\bf  Proof of (\ref{a1})}\label{A1}
We claim that 
\begin{equation} \label{aa1}
Z_{st} \otimes Z_{st} = (\partial_XZ_s \otimes \partial_XZ_s) X_{st} \otimes X_{st} + R^{10}_{st} 
\end{equation}
where $| R^{10}_{st} |\leq C|t-s|^{3\alpha}$.
Indeed, as $(Z, \partial_X Z)\in  \mathcal{D}_X^{2\alpha}([0,T],U)$ we have that 
\[
Z_{st}= \partial_X Z_z X_{st} + R^Z_{st}.
\]
Then
\begin{equation}
R^{10}_{st}= R^Z_{st}\otimes R^Z_{st} + R^Z_{st}\otimes \partial_XZ_s X_{st}+\partial_XZ_s X_{st}\otimes R^Z_{st}.
\end{equation} 
that satisfies $| R^{10}_{st} | \leq C|t-s|^{3\alpha}$. This proves the claim.

Applying $\frac{1}{2}D^2g(t, Z_s)$ into (\ref{aa1}), we have 
\begin{eqnarray}
\frac{1}{2}D^2g(t,Z_s)Z_{st}\otimes Z_{st} & = & D^2g(t,Z_s)(\partial_XZ_s \otimes \partial_X Z_s )\frac{1}{2}X_{st}\otimes X_{st} +\nonumber \\ & & + \frac{1}{2}D^2g(t, Z_s) R^{10}_{st} \label{eqa1}
\end{eqnarray}
Since 
\[
X_{st} \otimes X_{st}=\frac{1}{2}[\mathbf{X}]_{st}+ Sym \mathbb{X}_{st}
\]
and the  symmetry of $D^2g(t, Z_s)$, we have
\begin{eqnarray}\label{sim1}
D^2g(t,Z_s)(\partial_XZ_s \otimes \partial_X Z_s )\frac{1}{2}X_{st}\otimes X_{st}& = &D^2g(t,Z_s)(\partial_XZ_s \otimes \partial_X Z_s )\mathbb{X}_{st}+ \nonumber \\
 & &  +  \frac{1}{2}D^ 2g(s, Z_s)\partial_XZ_s \otimes \partial_X Z_s  \nonumber \\ & &  [ \mathbf{X} ]_{st} 
\end{eqnarray}

Substituting (\ref{sim1}) into (\ref{eqa1}) yields
\begin{eqnarray}
\frac{1}{2}D^2g(t,Z_s)Z_{st}\otimes Z_{st} & = & D^2g(t,Z_s)(\partial_XZ_s \otimes \partial_X Z_s )\mathbb{X}_{st} +\nonumber \\& & + \frac{1}{2}D^ 2g(s, Z_s)\partial_XZ_s \otimes \partial_X Z_s  [ \mathbf{X} ]_{st} +
\\ & & + \frac{1}{2}D^2g(t, Z_s) R^{10}_{st}. \label{eqa2}
\end{eqnarray}
Using that $D^2g \in C(U, \mathcal{C}^{\alpha}([0,T], W))$ we have that  
\begin{equation} \label{eqa3} 
R^ {11}_{st}:= (D^2g(t,Z_s)-D^2g(s,Z_s))(\partial_XZ_s \otimes \partial_X Z_s )\mathbb{X}_{st}
\end{equation}
Thus
\begin{eqnarray}
\frac{1}{2}D^2g(t,Z_s)Z_{st}\otimes Z_{st} & = & D^2g(s,Z_s)(\partial_XZ_s \otimes \partial_X Z_s )\mathbb{X}_{st} + \nonumber \\
& &+ \frac{1}{2}D^ 2g(s, Z_s)\partial_XZ_s \otimes \partial_X Z_s  [ \mathbf{X} ]_{st} + R^1_{st}
\end{eqnarray}
where $R^1_{st}= \frac{1}{2}D^2g(t, Z_s) R^{10}_{st} + R^ {11}_{st} $. Clearly, $|R^ {11}_{st}| \leq C|t-s|^ {3\alpha}$.

\subsection{\bf  Proof of (\ref{a2})}\label{A2}

We can take $\partial_X Dg(\cdot, x)=D\partial_Xg(\cdot, x)^*$, see Appendix \ref{A3}. Then
\begin{eqnarray}
Dg(t,x)(u)& = & Dg(s, x)(u) + \partial_X Dg(s,x)(u) X_{st} + R^{\prime}_{st}(u) \nonumber \\
& = & Dg(s, x)(u) + D\partial_X g(s,x)^* (u) X_{st} + R^{\prime}_{st}(u) \nonumber \\
& = & Dg(s, x)(u) + D\partial_X g(s,x) X_{st}(u) + R^{\prime}_{st}(u) \label{a21}
\end{eqnarray}
where $|R^{\prime}_{st}|\leq |t-s|^{2\alpha}$.

Substituting $u=Z_{st} $ into  (\ref{a21}) we obtain 
\begin{eqnarray}
Dg(t,x)Z_{st}& = & Dg(s, x)Z_{st} + D\partial_X g(s,x) X_{st}(Z_{st}) + R^{\prime}_{st}(Z_{st}) \nonumber \\
& = & Dg(s, x)Z_{st}+ D\partial_X g(s,x) X_{st}(\partial_XZ_s X_{st}) + R^{2}_{st} \nonumber \\
& = & Dg(s,Z_s)Z_{st} +(Dh(s,Z_s) \circ \partial_X Z_s )X_{st} \otimes X_{st} + \nonumber \\ & & +R^2_{st} 
\end{eqnarray}
where $R^2_{st}= R^{\prime}_{st}(Z_{st})+ D\partial_X g(s,x) X_{st}(R^ Z_{st}) $.

We have that $ |R^2_{st}| \leq C|t-s|^{3\alpha}$, because
\begin{eqnarray}
|R^2_{st}|  & \leq & |R^{\prime}_{st}(Z_{st})|+ |D\partial_X g(s,x) X_{st}(R^ Z_{st})| \nonumber \\ 
& \leq  & |R^ {\prime}_{st}| |Z_{st}| + \sup_{s,x}  |D\partial_X g(s,x)| |X_{st}| |R^ Z_{st}| \nonumber \\
& \leq & C|t-s|^{3\alpha} +\sup_{s,x}  |D\partial_X g(s,x)| C |t-s|^{3\alpha}.
\end{eqnarray}

\subsection{\bf  Proof of (\ref{com})}\label{A3}
We claim that,  
\begin{equation}
Dg(t,x)\cdot u=\int_0^tDh(r,x)\cdot u d\mathbf{X}_r
\end{equation}
for all $(t, x)\in [0,T] \times U$ and $u \in U$.

In fact, 
\begin{eqnarray}\label{a31}
\|\frac{g(t, x+\varepsilon u )-g(t,x)}{\varepsilon}-\int_0^tDh(r,x) \cdot u d\mathbf{X}_r\|_{\mathcal{D}_X^{2\alpha}}& = & \|\frac{1}{\varepsilon}\int_0^t (h(r,x+\varepsilon u)-h(r,x))d\mathbf{X}_r - \nonumber \\
& & -\int_0^tDh(r,x) \cdot u d\mathbf{X}_r\|_{\mathcal{D}_X^{2\alpha}} \nonumber \\
& = &\| \int_0^t   (\int_0^1 Dh(r, x+s\varepsilon u)\cdot u ds -   \nonumber  \\ & & -Dh(r,x)\cdot  u ) d\mathbf{X}_r\|_{\mathcal{D}_X^{2\alpha}}
\end{eqnarray}
and we have that the following inequality  holds
\begin{eqnarray}
\|  \int_0^1 Dh(r, x+s\varepsilon u)\cdot u ds -Dh(r,x)\cdot  u \|_{\mathcal{D}_X^{2\alpha}} & \leq & \int_0^1\| Dh(r, x+s\varepsilon u)\cdot u -\nonumber \\ & &-Dh(r,x)\cdot  u \|_{\mathcal{D}_X^{2\alpha}} ds. \label{a32}
\end{eqnarray} 
Combining (\ref{a31}), (\ref{a32}), $(Dh, \partial_X Dh) \in C(U, \mathcal{D}^{2\alpha}([0,T],L(U,L(V,W))))$ and the stability of the rough integral (see \cite{Friz Hairer}, Theorem 4.17) we obtain that 
\begin{equation}
\lim_{\varepsilon \rightarrow 0}\|\frac{g(t, x+u )-g(t,x)}{\varepsilon}-\int_0^tDh(r,x) \cdot u d\mathbf{X}_r\|_{\mathcal{D}_X^{2\alpha}}=0
\end{equation}
Thus $Dg(t,x)\cdot u=\int_0^tDh(r,x)\cdot u d\mathbf{X}_r$ as claimed.

We can choose the Gubinelli derivative $\partial_X Dg(\cdot, x)$ as $D\partial_Xg(\cdot, x)^*$. This implies that
\begin{equation}
((\partial_X Dg)(s, Z_s)\cdot \partial_XZ_s)\mathbb{X}_{st}  =  ((D \partial_X g)(s, Z_s)\circ  \partial_XZ_s)\mathbb{X}_{st}^*.
\end{equation}

\section*{Acknowledgements}

This study was financed in part by the Coordena\c{c}\~ao de Aperfei\c{c}oamento de Pessoal de N\'ivel Superior – Brasil (CAPES) – Finance Code 001, (Process number: Math-Amsud- 88887.333669/2019-00) -- on behalf of the first-named author -- and by grant 2020/04426-6, S\~ao Paulo Research Foundation (FAPESP) -- on behalf of the second-named author.


\end{document}